\documentclass[11pt]{amsart}
\usepackage[utf8]{inputenc}

\usepackage{enumitem}
\usepackage{amsfonts}
\usepackage{amsmath}
\usepackage{amsthm}
\usepackage{amssymb}
\usepackage{mathrsfs}
\usepackage{enumitem}
\usepackage{forest}
 
\usepackage[all]{xy}
\usepackage{appendix}

\usepackage{tikz}
\usetikzlibrary{shapes.geometric}
\usetikzlibrary{positioning}
\usetikzlibrary{automata}
\usetikzlibrary{decorations}
\usetikzlibrary{decorations.pathreplacing}

\usepackage{tikz-cd}

\theoremstyle{definition}
\newtheorem{defi}{Definition}[subsection]
\newtheorem{eg}[defi]{Example}
\newtheorem{rem}[defi]{Remark}

\newtheorem{cor}[defi]{Corollary}

\newtheorem{theo}[defi]{Theorem}

\newtheorem{fact}[defi]{Fact}

\newtheorem{defi2}{Definition}[section]
\newtheorem{eg2}[defi2]{Example}
\newtheorem{rem2}[defi2]{Remark}

\newtheorem{cor2}[defi2]{Corollary}
\newtheorem{lem2}[defi2]{Lemma}
\newtheorem{theo2}[defi2]{Theorem}
\newtheorem{con2}[defi2]{Conjecture}

\newcommand{\C}{\mathsf{C}}

\renewcommand{\P}{\mathcal{P}}
\newcommand{\Q}{\mathcal{Q}}
\newcommand{\Colors}{\operatorname{Col}}

\newcommand{\fc}{\mathfrak{fc}}

\newcommand{\vect}{\mathsf{grVect}}
\newcommand{\ch}{\mathsf{Ch}(\mathsf{grVect})}

\newcommand{\lef}{\operatorname{left}}
\newcommand{\rig}{\operatorname{right}}

\newcommand{\Free}{\operatorname{Free}}

\title{From polytopes to operads and back}

\author{Sergey Arkhipov, Daria Poliakova}

\begin{document}
\maketitle

\section*{Introduction}
Similarly to the well-known case of the category of modules over an algebra, for the category of algebras over an operad to be monoidal, the operad has to be endowed with a Hopf structure, aka an operadic diagonal. \\

In homotopical algebra, a number of DG-operads (or operadic bimodules) have topological origins: they come from topological operads (or operadic bimodules) via the functor of singular or cellular chains. Notice that for a CW-complex, its diagonal map is continuous but not cellular unless the CW-complex is the point. Yet, for explicit computations cellular chains are better than singular chains because the resulting complexes are much smaller. However, diagonals for these DG-operads/bimodules become nontrivial to find.  \\

In what follows,  CW-complexes that encode families of algebraic operations  are contractible (i.e. we explore different incarnations of associativity), moreover we restrict our interest to families of polytopes. Operadically interesting families of polytopes include simplices, cubes, freehedra \cite{San} \cite{Pol}, associahedra \cite{4Sta}, multiplihedra \cite{Sta2} \cite {For}... For all of the families above, cellular diagonals are known. The diagonal for simplices is Alexander-Whitney diagonal responsible for multiplication in singular cohomology, and the diagonal for cubes is Serre diagonal responsible for multiplication in singular cubic cohomology.  A complicated general result is Saneblidze--Umble's construction of diagonal for {\em permutahedra} \cite{4SU}: in fact, all of the polytope families above are generalized permutahedra in the sense of Postnikov \cite{Pos}, and their diagonals are inherited from the ones for permutahedra via projections.  However, for polytopes more difficult than simplices and cubes, these diagonals fail to be strictly coassociative. \\

This naturally brings us to a question: to what higher structure do these diagonals fit? This was a question studied by Loday for associahedra in his later years (as was communicated to us by Dotsenko; the results were not published). The first guess would be that for every polytope $P$, its cellular complex is an $A_\infty$-coalgebra. This does not seem to be a plausible form of the final answer for operadically meaningful families of polytopes: $A_\infty$-coalgebras fail to form a strictly monoidal category, thus the notion of an operad in $A_\infty$-coalgebras does not make sense. (This is the origin of the well-known difficulty to define an explicit coherent homotopy monoidal structure on the category of $A_\infty$-algebras.) \\

In this paper, we propose a road map to remedy the failure.
\begin{itemize}
    \item 
  The notion of $A_\infty$-coalgebra on a complex $V$ is upgraded to a richer structure encoded into {\em an involution automorphism} $I(t)$ of the completed free algebra $\hat{T}(V)[[t]]$ (compare: an $A_\infty$-coalgebra structure on $V$ is a degree 1 homological vector field on $\hat{T}(V)$, which can be seen as the linear term of $I(t)$ as above). Call such structure an {
 \it integrated $A_\infty$-coalgebra} (work over $\mathbb{F}_2$ to avoid sign issues).
 \item The category of integrated $A_\infty$ coalgebras carries a natural monoidal structure upgrading tensor product of complexes, thus it makes sense to consider operads in this monoidal category. 
  \item
  For a good enough polytope $P$, its cellular chains $C_*(P)$ should be made into an integrated $A_\infty$-coalgebra. In particular, the linear term of $I(t)$ provides an $A_\infty$-coalgebra structure on $C_*(P)$.
  \end{itemize}
  
  It is completely unclear from the steps above where the data for the involution $I(t)$ are supposed to come from. Here is the striking source for such structure: multicategorification. \\
  
  It is time to recall the combinatorics of polytopal diagonals for simplices, cubes, freehedra and associahedra. Their 1-skeletons are directed graphs with no cycles, which provides a (non-reflexive) partial order on all faces. Then diagonals are given by what Loday calls {\em magic formulas}:

$$ \Delta: F \mapsto \sum_{\substack{ F_1, F_2 \subset F \\ F_1 \leq F_2 \\ \dim F_1 + \dim F_2 = \dim F}} \pm F_1 \otimes F_2 $$

The diagonals defined by magic formulas are straightforwardly compatible with operadic compositions, but Leibniz rule is a lot of cancellations. What is the combinatorial reason for all these cancellations to work? When will there be similar magic formulas for all the operations in the integrated $A_\infty$-coalgebra? \\

  The proposed answer involves constructing a colored operad with colors enumerated by faces of the polytope, and viewing the diagonal as a component of the Poincar\'e-Hilbert series of this operad. We make use of the following fact: given a colored DG-operad $\P$ with colors forming a basis in a complex $C$ which is {\it Koszul self-dual}, the Poincar\'e-Hilbert series of $\P$ provide an endomorphism $I(t)$ of $\hat{T}(C)((t))$ which is an involution. Our main conjecture states that good polytopes (satisfying the condition we call {\it shortness}) give rise to Koszul self dual operads.  In the present paper the conjecture is verified for simplices, cubes, and for all polygons. \\
  
  Of operadically meaningful polytopes, shortness is verified for freehedra in the sequel of the current paper. Unfortunately, associahedra fail to be short, therefore the construction of colored operad as explained above should be modified in their case. This modification is the goal of an ongoing project.\\
  
  Let us outline the structure of the paper. In Section 1, we present generalities on colored operads, their Poincar\'e-Hilbert series, their duality and their Gr\"{o}bner bases. In Section 2 we introduce our main construction of a colored operad associated to a directed polytope. In Section 3 we explain that our construction is a monoidal functor. In Section 4 we prove Koszulity and self-duality for colored operads associated to simplices and cubes. In Section 5 we prove the same results for arbitrary polygons. In Section 6 we explain the generality in which we conjecture Koszulity and self-duality, by introducing shortness; and define integrated $A_\infty$-coalgebras over $\mathbb{F}_2$ (that is, without signs).
  
  \subsubsection*{Acknowledgements} We are grateful to Vladimir Dotsenko for useful discussions, in particular for first suggesting that the equation $I^2 = \operatorname{Id}$ hints at the presence of a self-dual operad. We are also grateful  Vladimir Baranovsky, Ryszard Nest and Jim Stasheff for useful discussions.

\section{Colored operads, Poincar\'e-Hilbert endomorphisms and duality}

\subsection{Generalities}
Recall the terminology and notation for colored operads. For a fixed monoidal category $\C$ and for a fixed set of colors $\Colors$ (which we assume to be finite), ${\mathbb{N}}$-$\operatorname{Seq}_{\operatorname{Col}}(\C)$ is the category of colored $\mathbb{N}$-sequences in $\C$, where an object $\P$ is a collection of $\P(c_1,\ldots,c_k;c) \in \C$ for all tuples $c_1,\ldots, c_k, c$ with $c_i$ and $c$ in $\Colors$. The colors $c_i$ are called inputs and the color $c$ is called output. For $\P$ and $\Q$ in ${\mathbb{N}}$-$\operatorname{Seq}_{\operatorname{Col}}(\C)$, their tensor-product $\P \odot \Q$ is given by
\begin{align*}
&    (\P \odot \Q)(c_1, \ldots, c_n;c) = \\
&    \bigoplus_{\substack{i_1+ \ldots +i_k = n \\ c'_1, \ldots, c'_k \in \operatorname{Col}}} \P(c'_1,\ldots,c'_k;c) \otimes \Q(c_1,\ldots, c_{i_1};c'_1) \otimes \ldots \otimes \Q(c_{n-i_k+1},\ldots, c_n; c'_k)
\end{align*} 

A colored operad is a unital algebra in ${\mathbb{N}}$-$\operatorname{Seq}_{\operatorname{Col}}(\C)$. Given unitality, we will often express operadic structure through {\em elementary} compositions:
\begin{align*} 
&    \circ_{i} \colon \P(c'_1, \ldots, c'_m; c_i) \otimes \P(c_1, \ldots, c_n; c) \to \\
&    \P(c_1, \ldots, c_{i-1}, c'_1, \ldots, c'_m, c_{i+1}, \ldots, c_n; c)
\end{align*} 
Colored cooperads are defined dually. \\

We mostly work in  $\C = \vect$,  the monoidal category of graded vector spaces that are bounded from below and are finite dimensional in every graded component.  Notice that the grading in question is the {\em inner} one and has nothing to do with homological grading to be introduced later. \\

For a colored operad $\P$ in $\vect$, let $T(\P)$ be the algebra of formal power series in non-commuting variables $c \in \Colors(\P)$, with coefficients in $k((t))$ (this means that the extra invertible variable $t$ is required to commute with everything). For a graded vector space $V = \bigoplus V_i \in \vect$ let $\dim V$ be the Laurent series in $t$ whose coefficient near $t^n$ is the dimension of $V_n$. Let $\operatorname{Tuples}(\P)$ be the set of all ordered tuples of elements in $\Colors(\P)$ (repetitions are allowed) - these are possible inputs of operations. For $\overset{\to}{u} \in \operatorname{Tuples}(P)$, let $|\overset{\to}{u}|$ denote the length of the tuple. Notice that any such tuple can be viewed as a non-commutative monomial in $T(\P)$.

\begin{defi}
The Poincar\'e-Hilbert endomorphism for $\P$ is an endomorphism $f_\P$ of $T(\P)$, which is defined on generators by 
$$f_\P(c) = \sum_{ \overset{\to}{u} \in \operatorname{Tuples}(\P)} \left(t^{|\overset{\to}{u}|-1}\right) \left[\dim P(\overset{\to}{u};c) \right] \overset{\to}{u} $$
and extended to an algebra map by multiplication since the algebra in question is free.
\end{defi}

Notice that the factor  $t^{|\overset{\to}{u}|-1}$ ensures that the resulting power of $t$ encodes the {\em total} gradings of operations, which are sums of inner gradings and arity gradings. \\

Next we extend our setting to the monoidal category of {\em complexes} in graded vector spaces $\ch$. We follow the standard conventions to produce the total homological (resp. the total inner) grading on the tensor product. \\

For  such complex $C = \bigoplus C^i$, we define its dimension as $\operatorname{Dim}(C)= \sum (-1)^{i} \dim C^i$, whenever this Laurent series in $t$ is well defined, i.e  the coefficient by each power of $t$ is a finite sum.\\

The definition of Poincar\'e-Hilbert endomorphism is repeated verbatim for (co)operads in $\ch$. \\

We recall the (reduced)  Bar construction  for colored operads following i.e. \cite{4AK}.

\begin{defi}
A {\em marked tree} is a planar rooted tree where each edge is decorated with one of the colors $c \in \Colors$. Note that inner vertices can have any positive number of incoming edges, including one (stems are allowed). Let $\operatorname{Tree}(c_1, \ldots, c_n;c)$ denote the (infinite) set of marked trees where the leaf edges are decorated by $c_i$ and the root edge is decorated by $c$. For an inner vertex $v$ of such a tree $T$, let $\operatorname{In}(v)$ denote the string of colors decorating the incoming edges of $v$, left to right.
\end{defi}

\begin{defi}
For a colored $\mathbb{N}$-sequence $\P$ in a base monoidal category $\mathsf C$, the  colored $\mathbb{N}$-sequence $F(\P)$ is defined as
\[ \Free(\P)(c_1, \ldots, c_n; c)=\bigoplus_{T \in \operatorname{Tree}(c_1, \ldots, c_n;c)} \bigotimes_{v \in T} \P(\text{In}(v); \text{Out}(v)). \]
This sequence  comes equipped with a map $\Free(\P) \odot \Free(\P) \to \Free(\P)$ providing a structure of the {\em free operad} on $P$, and with a map $\Free(\P) \to \Free(\P) \odot \Free(\P)$ giving it the structure of {\em the cofree cooperad} on $\P$.
\end{defi}

An {\em tree monomial} is a tensor monomial of $\bigotimes_{v \in T} \P(\text{In}(v); \text{Out}(v))$ for some fixed tree $T$. A tree monomial is called quadratic if the corresponding tree has 2 inner vertices. An homogeneous element of the free colored $\mathbb{N}$-sequence is called quadratic if it is a sum of quadratic tree monomials. \\

Now let $\P$ be a colored operad that is {\em augmented} in the following sense: $$ \P \simeq R \oplus \widetilde{\P}$$ where $R$ is the semisimple algebra $$R \simeq \bigoplus_{c \in \Colors(\P)} k \cdot 1_c$$ 
In the present paper, we work with colored operads that allow unary operations in $\widetilde{\P}$.\\ 

For the colored $\mathbb{N}$-sequence $\widetilde{\P}$, we define its suspension  $s \widetilde{\P}$ by 
$$(s \widetilde{\P})(c_1, \ldots, c_n; c)^i= \widetilde{\P}(c_1, \ldots, c_n; c)^{i+1}$$ 
and 
$$d_{s \widetilde{\P}}(v)=(-1)^{|v|}s (d_{\widetilde{\P}} (v))$$

\begin{defi}
For the Bar cooperad $\operatorname{Bar}(\P)$, its underlying colored sequence is $R \oplus \Free(s \widetilde{\P})$. Its differential is obtained by adding a term to the differential on the cofree cooperad, where this extra term encodes the operadic compositions in $\P$ (see 6.3 in \cite{4AK}).
\end{defi}

In the next subsection we show that under appropriate finiteness conditions on $\P$, Poincar\'e-Hilbert endomorphisms of $\P$ and $\operatorname{Bar}(\P)$, after some modification of signs, become composition-inverse to each other.

\subsection{Series inversion}
We begin from recalling the well-known formula for series inversion. This is an interpretation of the Faa di Bruno formula, with a proof available e.g. in \cite{4AA}. We also present the proof here, because we need to generalize it later.\\

Let $f$ be an endomorphism of $k[[x]]$, given by sending $x$ to the power series $f(x) = x + f_2 x^2 + \ldots$ and extending multiplicatively.  We would like to obtain an explicit description for the coefficients of $g$, its composition inverse. The formula is stated in terms of planar trees $T$ that are only allowed to have inner vertices of valency $3$ or higher. \\

Denote by $\operatorname{Tree}(n)$ the (finite) set of such trees with $n$ leaves. For a tree with $n$ leaves, set $|T| = n - \#\text{inner edges}-2$. For any tree $T$, denote by $f_T$ the product of $f_i$, where $i$ goes through all the corollas of $T$. \\

\begin{theo}
\label{invert}
The inverse endomorphism $g$ has coefficients $g_1 = 1$ and, for $n \geq 2$,
$$g_n = \sum_{T \in \operatorname{Tree}(n)}(-1)^{|T|+1}f_T$$
\end{theo}

\begin{proof}
Consider the result of evaluating $f \circ g$ on $x$:
$$(x+g_2x^2+g_3x^3+\ldots)+f_2(x+g_2x^2+\ldots)^2+f_3(x+\ldots)^3+\ldots$$

The coefficient by $x$ is equal to $1$. For $g$ to be inverse of $f$, all other coefficients in the expression above need to vanish. The coefficient of $x^n$ is clearly equal to the sum of $f_k \cdot g_{i_1} \cdot \ldots \cdot g_{i_k}$, for all expressions $i_1+\ldots+i_k = n$ (where $f_1$ and $g_1$ should be read as $1$). \\

We verify the base of induction: the coefficient near $x^2$ is $g_2+f_2$, so $g_2 = -f_2$, which is in agreement with the trees formula (the single 2-leaved tree has $|T| = 2-0-2 = 0$). Assume that for $k < n$ the above formula for $g_k$ is proved. \\

By evaluating the coefficient near $x^n$ we obtain that 
$$ g_n = f_1 \cdot g_n = -\sum_{\substack{ k \geq 2, \\ i_1+\ldots+i_k = n} }f_k \cdot g_{i_1} \cdot \ldots \cdot g_{i_k}$$
where ${i_s}<n$ for every $s$, so we can substitute the values given by inductive assumption. Then we are left to observe that any tree with $n$ leaves can be uniquely constructed out of $k \geq 2$ trees with $i_1$ to $i_k$ leaves, by adding an edge to the root of each of these smaller trees and gluing at the bottom: 

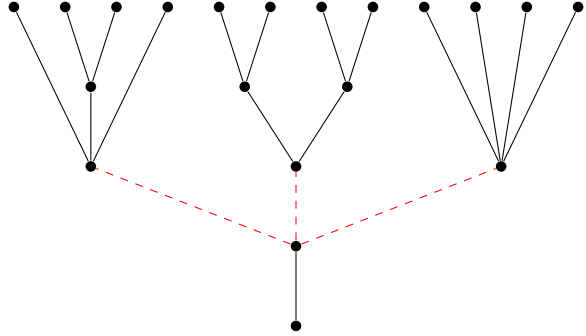
\begin{center}
\begin{forest}
for tree = {grow'=90,circle, fill, minimum width = 4pt, inner sep = 0pt, s sep = 15pt}
[[[ {}, edge = {red, dashed} [{}, tier = 3][[{}, tier = 3][]][{}, tier = 3]] [{}, edge = {red, dashed} [[][]][[][]]] [{}, edge = {red, dashed}[{}, tier = 3] [{}, tier = 3][{}, tier = 3] [{}, tier = 3]] ]]
\end{forest}
\end{center}

We denote this gluing operation by $*$. Then for $T = T_1 * \ldots * T_k$ we have $|T| = |T_1| + \ldots + |T_k| + k - 2$, so the signs match, which finishes the proof.

\end{proof}

Below we generalize the theorem above in the following directions: 

\begin{itemize}
    \item allow more variables that do not commute 
    \item allow the linear part to be non-identity. \\
\end{itemize}

Let $f$ be an endomorphism of $T(\P)$ as in the previous section. Denote by $f^{c}_{c_1\ldots c_n}$ the coefficient by the non-commutative monomial $c_1 \ldots c_n$ in the series $f(c)$ (this coefficient is itself a Laurent series in $t$).\\

Let $F$ be the matrix consisting of $f^c_{c'}$. For the endomorphism $f$ to have a composition inverse, its linear part $F$ clearly has to be invertible as a matrix. We additionally assume that $F$ can be written in the form $E + \widetilde{F}$, where $E$ is the identity matrix and $\widetilde{F}$ has only positive powers of $t$. Then, by elementary algebra, $F^{-1}$ can be written as $\sum_{n \geq 0} (- \widetilde{F}  )^n = E -  \widetilde{F} +  \widetilde{F}^2 \ldots$.\\

As before, we want an explicit formula for the coefficients of the inverse endomorphism $g$. This formula is stated in terms of marked trees decorated with out variables. For $T \in \operatorname{Tree}(c_1, \ldots, c_n;c)$, let $f_T$ be the product of $f^{\operatorname{Out}(v)}_{\operatorname{In}(v)}$ over all the inner vertices $v$ of $T$, where $\operatorname{Out}(v)$ is the variable decorating the output edge of $v$, and $\operatorname{In}(v)$ is the monomial obtained by multiplying together the variables that decorate input edges of $v$, left to right -- with one exception: when $|\operatorname{In}(v)| = 1$, use $ \widetilde{f}^{\operatorname{Out}(v)}_{\operatorname{In}(v)}$.

\begin{theo}
Let $g$ be the composition inverse of $f$. Then its coefficient $g^c_{c_1 \ldots c_n}$ can be computed by the following formula:
$$g^c_{c_1\ldots c_n} = \sum_{T \in \operatorname{Tree(c_1,\ldots,c_n; c)}}(-1)^{|T|+1}f_T$$
\end{theo}

\begin{proof}
We first deal with linear parts. From the discussion above we know that $$G =  E -  \widetilde{F} +  \widetilde{F}^2 - \widetilde{F}^3 \ldots $$ so, at each place $$g^{c'}_c = \delta(c,c')-\widetilde{f}^{c'}_c + \sum_{c_1 \in \Colors(\P)} \widetilde{f}^{c'}_{c_1} \widetilde{f}^{c_1}_c - \sum_{c_1,c_2 \in \Colors(\P)} \widetilde{f}^{c'}_{c_1} \widetilde{f}^{c_1}_{c_2} \widetilde{f}^{c_2}_c + \ldots $$
which is precisely the marked trees formula where all involved trees are stems and thus all vertices are binary.\\ 


We now proceed by induction on the number of leaves. Assume that for $k < n$, the formulas for $g^c_{c_1 \ldots c_m}$ are proved. We look at the coefficient near $x_{c_1} \ldots x_{c_n}$ at $g \circ f$ evaluated at $x_c$. This coefficient is equal to 
$$\sum_{c' \in \Colors(\P)}f^c_{c'} g^{c'}_{c_1 \ldots c_n} + \sum_{\substack{c'_1,\ldots,c'_k \in \Colors(\P) \\ n = i_1 + \ldots + i_k}} f^c_{c'_1\ldots c'_k} g^{c'_1}_{c_1 \ldots c_{i_1}} \ldots g^{c'_k}_{c_{n-i_k+1} \ldots c_n}$$
so for it to vanish, we need the following equality to hold
$$\sum_{c' \in \Colors(\P)}f^c_{c'} g^{c'}_{c_1 \ldots c_n} = - \sum_{\substack{c'_1,\ldots,c'_k \in \Colors(\P) \\ n = i_1 + \ldots + i_k}} f^c_{c'_1\ldots c'_k} g^{c'_1}_{c_1 \ldots c_{i_1}} \ldots g^{c'_k}_{c_{n-i_k+1} \ldots c_n}$$
We now fix $c_1$, $\ldots$, $c_n$ but allow $c$ to vary. Then the equations as above assemble into the following:
$$F\cdot g_{c_1\ldots c_n} = - \sum_{\substack{c'_1,\ldots,c'_k \in \Colors(\P) \\ n = i_1 + \ldots + i_k}} f_{c'_1\ldots c'_k} g^{c'_1}_{c_1 \ldots c_{i_1}} \ldots g^{c'_k}_{c_{n-i_k+1} \ldots c_n}$$
where $F$ is the matrix corresponding to the linear part of $f$, $g_{c_1\ldots c_n}$ is the vector with components $g^c_{c_1\ldots c_n}$, and $f_{c'_1\ldots c'_k}$ is the vector with components $f^c_{c'_1\ldots c'_k}$. We multiply both sides by $F^{-1}$, which was already shown to control the stems, and insert the coefficients of $g$ that we know by induction. Then on the left hand side we are left just with the vector $g_{c_1\ldots c_n}$ the entries of which we want to know, and the summands on the right hand side bijectively correspond to appropriately marked trees -- this can be seen by uniquely decomposing marked trees similarly to the decomposition of unmarked trees as in the proof of Theorem \ref{invert}.
\end{proof}

We obtain an immediate corollary for Poincar\'e-Hilbert endomorphisms  of Bar-dual operads and cooperads.

\begin{cor}
Let $I$ be an endomorphism of $T(\P)$ sending each variable $c$ to $-c$, and $t$ to $-t$. For a colored operad $\P$ and its Bar-dual cooperad $\operatorname{Bar}(\P)$, we have 
$ f_\P \circ I \circ f_{\operatorname{Bar}(\P)} \circ I = \operatorname{Id}$.
\end{cor}

Note that when all relevant dimensions are finite, a cooperad can be viewed as an operad, by dualizing. Therefore, we can speak of self-dual operads. If an operad $\P$ is self-dual, then the statement above implies that $f_\P \circ I$ is an involution. Throughout the paper, this will be referred to as {\em involutive property} of $f_\P$.

\subsection{Koszul duality}
If operads are sufficiently nice (namely, Koszul), it is possible to replace Bar duality by quadratic duality, where the latter notion is way more computable. The foundational work here was \cite{4GK}, although there the authors deal with symmetric operads in one color. \\

Below we modify the definitions from (\cite{LV}, 7.2) for the case of non-symmetric colored operads.

\begin{defi}
For a colored $\mathbb{N}$-sequence $V$ with finite-dimensional complexes $V(s_1, \ldots, s_n;t)$ and its free colored $\mathbb{N}$-sequence $\Free(V)$, let $\Free(V) \{ n \} $ be the subspace of $\Free(V)$ spanned by tree monomials where trees have $n$ inner vertices.
\end{defi}

Notice that for every $(s_1, \ldots, s_n;t)$ $\Free(V)(s_1, \ldots, s_n;t)$ is also finite-dimensional. Now we can recall quadratic operads and their quadratic duals.

\begin{defi}
An operad $\P$ is quadratic if it is realized as $\Free(V)/R$, where $V$ is some colored $\mathbb{N}$-sequence and $R$ is an operadic ideal of $\Free(V)$ generated by its intersection with $\Free(V)\{ 2 \}$, or, explicitly,
$$ R \{2 \} = \bigoplus_{(s_1,\ldots, s_n;t)} R\{2\} (s_1, \ldots, s_n; t)$$
where $R\{2\} (s_1, \ldots, s_n; t) \subset \Free(V) \{2 \} (s_1, \ldots, s_n;t)$. Such an operad is denoted by $\P(V, R \{2 \})$.
\end{defi}

\begin{defi}
Given a quadratic colored operad $\P=(V,R\{2\})$, we define the quadratic dual operad $\P^{!}$ as the quotient of $\Free(V^*)$ by the ideal of relations $R^{!}$ generated by $R^{!} \{2 \} = \bigoplus_{(s_1, \ldots, s_n;t)} R^{!} \{2 \} (s_1, \ldots, s_n; t)$ defined as follows. Notice that $\Free(V^*) \{2 \} (s_1, \ldots, s_n,t)$ is canonically isomorphic to $\Free(V) \{2 \} (s_1, \ldots, s_n; t)^*$. Now take $$R^{!} \{2 \} (s_1, \ldots, s_n; t) = R \{2 \} (s_1, \ldots, s_n; t)^{\perp}. $$
\end{defi}

Similarly to Koszul duality for quadratic algebras, we need to introduce homological grading shifts. Let 
$\P=(V,R \{2 \})$ be a colored quadratic operad. Its de-suspention is the colored quadratic operad $s^{-1}\P=(s^{-1}V,s^{-2}R\{2\})$.\\

Define the sign operad (with one color) as follows: 
$${\mathcal S}(n)=\operatorname{Hom}((s\mathbb{K})^{\otimes n},s\mathbb{K}),$$
the operadic compositions are evident. Denote the linear dual cooperad by ${\mathcal S}^c$. \\

For $\P$ a colored $\mathbb{N}$-sequence and $\Q$ an $\mathbb{N}$-sequence in just one color, let $\Q \boxtimes \P$ be a colored $\mathbb{N}$-sequence given by 
$$\Q \boxtimes \P (c_1, \ldots, c_n;c) = \Q(n) \otimes \P (c_1, \ldots, c_n;c) $$
If both factors are operads in the corresponding categories the resulting colored sequence is also a colored operad.

\begin{defi} (see \cite{LV} 7.2.3, 7.2.4)
A colored quadratic operad $\P$ is Koszul if the cooperad  ${\mathcal S}^c\boxtimes (s^{-1}\P^{!})^*$ is quasiisomorphic to  $\operatorname{Bar}(\P)$.
\end{defi}

A useful tool for establishing Koszulity is the theory of Gr\"{o}bner bases, which we quickly recall. As an input, this theory must take some monomial order.

\begin{defi}
A monomial order is a linear order on the set of all tree monomials in some free operad $F(V)$. A monomial order is called admissible if it is
\begin{itemize}
    \item compatible with arities: for tree monomials $\alpha \in F(V)(c_1, \ldots, c_n;c)$ and $\beta \in F(V)(c'_1, \ldots, c'_m;c')$ we have $\alpha < \beta$ if $n < m$;
    \item compatible with compositions: if $\alpha \leq \alpha'$ and $\beta \leq \beta'$ then $\alpha \circ_i \beta \leq \alpha' \circ_i \beta '$ whenever these compositions are defined.
\end{itemize}
\end{defi}

Some standard admissible monomial orders are described in \cite{4DK} and \cite{4KK}. \\

Now assume that an operad $P$ is written as $F(V)/I$ where $F$ is the free operad on generators $V$ (of arbitrary arities), and $I$ is an operadic ideal. Fix some admissible monomial order.  The presence of this monomial order means that for any expression in our generators we can define its {\em leading term}, i.e. the greatest tree monomial that has a nonzero coefficient in this expression. For an expression $f$, its leading term will be denoted $\operatorname{lt}(f)$. For an operadic ideal $I$, $\operatorname{lt}(I)$ is the ideal generated by the leading terms of all elements of $I$.

\begin{defi}
A set of relations $G = \{ g_i \}$ is called a Gr\"{o}bner basis for $P = F(V)/I$, if $I = (G)$ and operadic ideals $(  \operatorname{lt}(G)  )$ and $\operatorname{lt}(I) = \operatorname{lt}((G))$. coincide. 
\end{defi}

We will make use of the following fact, see Corollary 3 in \cite{4DK} and Theorem 3.12 in \cite{4KK} for the colored case.

\begin{fact}
An operad with a quadratic Gr\"{o}bner basis is Koszul.
\end{fact}

There are different ways to check whether a set of relations forms a Gr\"{o}bner basis. If the dimensions of the components in the operad are already known, then there is a straightforward approach via normal forms.

\begin{defi}
\ \\ 
\begin{itemize}
    \item A tree monomial is a {\em normal form} with respect to $G$ if it is not divisible by a leading term of any $g \in G$.
    \item An arbitrary expression in $\Free(V)$ is a normal form if all its monomials are normal forms.
\end{itemize}
\end{defi}

It is true that (the images of) normal forms span the quotient $\Free(V)/(G)$, no matter if $G$ is a Gr\"{o}bner basis or not. However, if $G$ {\em is} a Gr\"{o}bner basis, the converse is also true.

\begin{fact}
\label{criterion}
$G$ is a Gr\"{o}bner basis if and only if the number of normal forms in every operadic component coincides with its dimension.
\end{fact}

\section{Main construction}
\subsection{Operad of a polytope.}
The central construction of the paper produces a colored operad from a polytope with certain additional combinatorial choices. 
\begin{defi}
A polytope $P$ is {\it directed} if its 1-skeleton is an oriented graph with no cycles, one source and one sink, and the same holds for every face of $P$.
\end{defi}

The conditions above hold for convex polytopes with direction coming from a linear functional, but we consider directed polytopes abstractly. Vertices of a directed polytope are partially ordered: $v_1 \leq v_2$ if there exists a directed edge-path from $v_1$ to $v_2$. This partial order can be extended to a non-reflexive operation on the faces of arbitrary codimension.

\begin{defi}
Let $F_1$ and $F_2$ be two faces of a directed polytope $P$. Then $F_1 \leq F_2$ if $\min F_1 \leq \max F_2$.
\end{defi}



Note that $F \leq F$ only holds when $F$ is a vertex. 

\begin{defi} A sequence of faces $(F_1, \ldots, F_n)$ is a {\it face chain} in a face $F$ if $F_i \subset F$ for any $i$ and $F_1 \leq \ldots \leq F_n$. The {\it excess} of $(F_1, \ldots, F_n)$ in $F$ is $(\dim F-1) - \sum (\dim F_i-1)$. The set of face chains in $F$ of length $n$ with excess $l$ is denoted by $\fc_l(F,n)$. \end{defi}

\begin{rem}
The notion of excess generalizes codimensions. Indeed, for a chain of length 1 the formula gives usual the codimension. Longer chains start having big codimensions if they are not fat enough for their length.
\end{rem}
Note that in general, excesses can be both positive and negative: for example, if a 3-dimensional polytope has a chain $F_1 < F_2 < F_3$ of 2-dimensional faces, then this chain would have excess $(3-1) - ((2-1) + (2-1) + (2-1)) = -1$. However, cases like this are unwelcome: the theory developed in this paper seems to work well precisely for the polytopes where excesses of nontrivial chains are strictly positive. \\

We now define $\mathcal{O}_P$, a colored operad in graded vector spaces associated to the directed polytope $P$.

\begin{defi}
The set of colors is given by all faces of $P$. The operation spaces are 

$$\mathcal{O}_P(F_1, \ldots, F_n; F) = \begin{cases} k[l-n+1] & (F_1, \ldots, F_n) \in \fc_l(F,n) \\ 0 & \text{else} \end{cases} $$

Thus the total grading of every operation is the excess of the corresponding chain. The composition maps are either $k[m] \simeq k[m]$ or $0  \to k[m]$.
\end{defi}

\begin{theo}
$\mathcal{O}_P$ is well-defined.
\end{theo}
\begin{proof}
We need to verify that we do not encounter a situation when the source of the composition map is nonzero while the target is zero. Once this is verified, the associativity of composition is a tautology.  We limit ourselves to pseudooperadic elementary compositions, and thus look at the map $\circ_i$:

$$\xymatrix { \mathcal{O}_P(F_1, \ldots, F_n; G_i) \otimes \mathcal{O}_P(G_1, \ldots, G_m; G) \ar[d] \\  \mathcal{O}_P(G_1, \ldots, G_{i-1},F_1, \ldots, F_n,G_{i+1}, \ldots, G_m; G) } $$

Suppose the source of this map does not vanish. This means that $(F_1,\ldots,F_n)$ is a face chain of $G_i$, and $(G_1, \ldots, G_m)$ is a face chain of $G$. Then in the target, the sequence $(G_1, \ldots, G_{i-1},F_1, \ldots, F_n, G_{i+1}, \ldots, G_m)$ is indeed a face chain of $G$. Its elements are clearly included in $G$ because the inclusions are composable. The inequality $G_{i-1} \leq F_1$ holds because we have $\max G_{i-1} \leq \min G_i$ from the face chain condition on $(G_1, \ldots, G_m)$ and $\min G_i \leq \min F_1$ because $F_1 \subset G_i$. The argument for $F_n \leq G_{i+1}$ is similar.
\end{proof}

Let us look at the smallest examples.

\begin{eg}
Let $P$ consist of just one point $x$. A chain where this point repeats $n$ times has excess $n-1$. So $\fc_{n-1}(x,n)$ consists of one chain and $\fc_l(x,n)$ is empty for $l \neq n-1$. This means that we have one operation of inner degree 0 in each arity, so $\mathcal{O}_P = Ass$. 
\end{eg}

\begin{eg}
Let $P$ be the interval $s$ with endpoints $x$ and $y$. Then $\fc_l(x,n)$ and $\fc_l(y,n)$ are the same as before. Inside $s$, there are face chains $x^i s y^j$ with $i+j = n$, $i \geq 0$ and $j \geq 0$, which have excess $n$, and there are face chains chains $x^iy^j$ for $i+j = n$, which also have excess $n$. Algebras over $\mathcal{O}_P$ consist of tuples $(A_x,M_s,A_y,[ \text{  }\text{  } ])$, where $A_x$ and $A_y$ are associative algebras, $M_s$ is an $A_x-A_y$ bimodule, and $[\text{  }\text{  }]$ is a bilinear form of inner degree $1$ on $A_x \otimes A_y$ with values in $M_s$. Also $\mathcal{O}_P$ has unary operations  $x \to s$ and $y \to s$, which correspond to maps $A_x \to M_s$ and $A_y \to M_s$ -- these are suggestively denoted by respectively  $[,1]$ and $[1,]$, though formally they are not expressed through the bracket because our algebras are not unital. 
\end{eg}

This latter example conveys the general flavour of our construction. For bigger polytopes $P$, algebras over corresponding operads will consist of multiple associative algebras and bimodules, with various bimodule-valued (not necessarily binary) brackets between them all. \\

Before moving on, let us take a look at the Poincar\'e-Hilbert endomorphisms of the operads in the examples above. For the point $x$, we have 

$$ f_x = x+tx^2+ t^2x^3 \ldots = \frac{x}{1-tx} $$

For the interval $s$ with endpoints $x$ and $y$, we have 

$$ f_x = \frac{x}{1-tx} $$
$$ f_y = \frac{y}{1-ty} $$
$$ f_s =  \frac{1}{1-xt}(s+(x+y-xy)t)\frac{1}{1-yt} $$

It can be checked by a direct computation that both these endomorphisms satisfy the involutive property. This brings us to a conjecture that for some polytopes, their operads are self-dual. We subsequently prove this conjecture for simplices, for products thereof, and for all polygons. In section \ref{algebra} we explain the generality in which we expect the conjecture to hold.

\subsection{Functoriality.}
We now describe the functoriality of our construction. \\

Let $Poly$ be the category whose objects are directed polytopes, and whose morphisms are inclusions that are injective on face posets and respect directions. The category $Poly$ is monoidal with respect to the obvious product. Let $ColOp$ be the category of colored DG-operads. It is monoidal with respect to the product that multiplies the color sets. \\

\begin{theo}
The assignment $O: P \mapsto \mathcal{O}_P$ forms a monoidal functor $Poly \to ColOp$. 
\end{theo}

\begin{proof}
An inclusion of directed polytopes $P \to Q$  gives a map of operads $\mathcal{O}_P \to O_Q$. For two polytopes $P_1$ and $P_2$, we have $O_{P_1 \times P_2} = O_{P_1} \times O_{P_2}$. Both statements are straightforward.
\end{proof}

\section{First examples: simplices and cubes}
In this section we consider the case of $P = \Delta_n$, the standard simplex. We prove that its operad $O_{\Delta_n}$ is Koszul, by constructing its quadratic Gr\"{o}bner basis. As already mentioned, Koszulity allows us to replace the computationally difficult notion of Bar-duality by a simpler notion of quadratic duality. We then observe quadratic self-duality of $O_{\Delta_n}$ and thus prove the involutive property of its Poincar\'e-Hilbert series. Functoriality from the previous section also implies the self-duality for operads corresponding to cubes. \\

We describe the operad $O_{\Delta_n}$ by generators and relations. In a simplex, faces correspond to all nonempty subsets $I \subset [0,n]$ -- these are the colors of our operad.

\begin{defi2}
A binary tree is {\em right-leaning} if for every inner vertex, its left incoming edge is a leaf.
\end{defi2}

\begin{lem2}
\label{generators}
Generating operations for $\mathcal{O}_{\Delta_n}$ are of two types: unary and binary. Generating unary operations are elementary inclusions $U_i(I)$ from color $I$ to color $I+i: = I\cup \{i \}$, where $i$ is some index outside $i$. Generating binary operations $B(I,J)$ are from colors $I$, $J$ to color $I+J: = I \cup J$, where $I$ and $J$ are subsets of indices such that $\max I = \min J$. They will be depicted as follows:
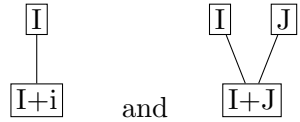
\begin{center}
\makebox[20mm]{\begin{forest}
for tree = {grow'=90, draw, minimum width = 4pt, inner sep = 2pt, s sep = 15pt}
[I+i  [I]]
\end{forest}} and \makebox[20mm]{\begin{forest}
for tree = {grow'=90, draw, minimum width = 4pt, inner sep = 2pt, s sep = 15pt}
[I+J [I][J]]
\end{forest}} 
\end{center}

\end{lem2}

\begin{proof}
Consider an arbitrary nonzero operation in $O_{\Delta_n}$. It goes from colors $I_1$, $\ldots$, $I_k$ to color $I$, where $\max I_s \leq \min I_{s+1}$ for every $s$. We express this operation through $U$ and $B$ as follows. At every color $I_s$, we apply unary operations appending, in descending order, all indices between $\max I_{s}$ and $\max I_{s-1}$ that are present in $I$ and not in $I_s$, up until the moment when we append $\max I_{s-1}$. Then we find ourselves in situation that all colors overlap, so we apply binary operations right to left. Thus our nonzero operation becomes represented by a right-leaning binary tree with stems attached to its leaves, suggestively called {\em the normal form}. 
\end{proof}

\begin{eg2}
Here is the normal form for the operation from colors $I_1 = 13$, $I_2 = 6$ and $I_3 = 6$ into the color $I = 12346$.

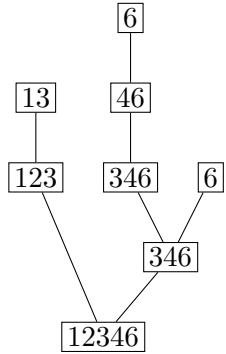
\begin{center}
\begin{forest}
for tree = {grow'=90, draw, minimum width = 4pt, inner sep = 2pt, s sep = 15pt}
[12346 [123, tier = 1[13]] [346 [346[46[6]]] [6, tier = 1] ]]
\end{forest}    
\end{center}
\end{eg2}

\begin{lem2}
\label{relations}
Relations on operations $U$ and $B$ are of the following types:
\begin{enumerate}
    \item for $I$, $J$, $K$ with $\max I = \min J$ and $\max J = \min K$, there is a relation $B(B(I,J),K)) - B(I,B(J,K))$:
    \begin{center}
 \makebox[20mm]{\begin{forest}
for tree = {grow'=90, draw, minimum width = 4pt, inner sep = 2pt, s sep = 15pt}
[I+J+K [I+J [I,tier = 1] [J]][K, tier = 1]]
\end{forest}} --             \makebox[20mm]{\begin{forest}
for tree = {grow'=90, draw, minimum width = 4pt, inner sep = 2pt, s sep = 15pt}
[I+J+K [I, tier = 1][ J+K [J, tier = 1][ K, tier = 1]]]
\end{forest}}
    \end{center}
    
    \item for $I$ and $J$ with $\max I < \min J$, there is a relation $B(U_{\min J}(I),J) - B(I,U_{\max I}(J))$:
    
    \begin{center}
    \makebox[20mm]{\begin{forest}
for tree = {grow'=90, draw, minimum width = 4pt, inner sep = 2pt, s sep = 30pt}
[I+J [I+min J [I,tier = 1]] [J, tier = 1]]
\end{forest}}
-- 
    \makebox[20mm]{\begin{forest}
for tree = {grow'=90, draw, minimum width = 4pt, inner sep = 2pt, s sep = 30pt}
[I+J[I, tier = 1][max I+J [J, tier = 1]]]
\end{forest}}
    \end{center}

    \item for $I$, $J$ and $i$ with $\max I = \min J$ and $i < \max I$, there is a relation $B(U_i(I),J) - U_i(B(I,J))$:
    
    \begin{center}
\makebox[20mm]{\begin{forest}
for tree = {grow'=90, draw, minimum width = 4pt, inner sep = 2pt, s sep = 15pt}
[i+I+J [I+J [I] [J]]]
\end{forest}} 
--
    \makebox[20mm]{\begin{forest}
for tree = {grow'=90, draw, minimum width = 4pt, inner sep = 2pt, s sep = 15pt}
[i+I+J [i+I [I,tier = 1] ][J, tier = 1]] 
\end{forest}}
    \end{center}

     \item for $I$, $J$ and $j$ with $\max I = \min J$ and $j > \max J$, there is a relation $U_j(B(I,J)) - B(I,U_j(J))$:
     
\begin{center}
    \makebox[20mm]{\begin{forest}
for tree = {grow'=90, draw, minimum width = 4pt, inner sep = 2pt, s sep = 15pt}
[I+J+j [I+J [I] [J]]]
\end{forest}} -- \makebox[20mm]{\begin{forest}
for tree = {grow'=90, draw, minimum width = 4pt, inner sep = 2pt, s sep = 15pt}
[I+J+j [I, tier = 1] [J+j [J, tier = 1]]]
\end{forest}}
\end{center}  

\item for $I$, $i$ and $j$ with $i < j$, there is a relation $U_i(U_j(I)) - U_j(U_i(I))$:

\begin{center}
\makebox[20mm]{\begin{forest}
for tree = {grow'=90, draw, minimum width = 4pt, inner sep = 2pt, s sep = 15pt}
[I+i+j [I+i [I]]]
\end{forest}} -- \makebox[20mm]{\begin{forest}
for tree = {grow'=90, draw, minimum width = 4pt, inner sep = 2pt, s sep = 15pt}
[I+i+j [I+j [I]]]
\end{forest}}
\end{center}

\end{enumerate}

\end{lem2}
\begin{proof}
The relations above are sufficient to bring any tree monomial to the normal form featured in the proof of Lemma \ref{generators}. Indeed, relations of type 3 and type 4 can be applied until the moment that all stems become attached to the top of the binary tree. Then relations of type 1 can be applied until the binary tree becomes right-leaning. Then relations of type 2 can be applied until stems become attached to the rightmost possible leaf of the binary tree. And finally, relations of type 5 can be applied until within every stem, the indices are appended in the descending order.
\end{proof}

To proceed with our proof of Koszulity, we now describe a monomial order on tree monomials consisting of operations $U$ and $B$. We first order the generating operations in the following way: 

\begin{itemize}
    \item All unary operations are smaller than all binary operations.
    \item $U_i(I)$ and $U_j(J)$ are first compared by the lengths of $I$ and $J$; if lengths coincide, then $I$ and $J$ are compared lexicographically; if $I = J$, then $U_i(I) < U_j(I)$ if $i>j$.
    \item $B(I,J)$ and $B(I',J')$ are first compared by the length of $I\cup J$ and $I' \cup J'$; if lengths coincide, then $I\cup J$ and $I' \cup J'$ are compared lexicographically; if they coincide, then the length of $I$ is compared to the length of $I'$.
\end{itemize}

For example, in  $O_{\Delta_1}$ the resulting order on generating operations is as follows:

$$ U_1(0) < U_0(1) < B(0,0)< B(1,1) < B(0,01) < B(01,1) $$

We now declare an order on tree monomials. To every tree monomial $T$ with $k$ leaves we associate a sequence $S(T)$ of $k$ words in the alphabet consisting of generating operations. The $i$th word in this sequence is obtained by going from $i$th leaf to the root and recording all operations that are encountered on the way. Tree monomials are then compared by their sequences: 

\begin{itemize}
    \item We first compare the number of words/leafs.
    \item If those coincide, we lexicographically compare vectors encoding word lengths.
    \item If those coincide, we lexicographically compare first words, then second words and so on until we encounter a difference. 

\end{itemize}

This order on tree monomials is a minor modification of the well-known {\em path-lexicographic order}, and can be easily checked to be admissible. \\

Note that the relations in Lemma \ref{relations} are already written in such a way that their leading terms go first. In the computations to follow, this will always be the way to write things.

\begin{theo2}
Relations from Lemma \ref{relations} form a quadratic Gr\"{o}bner basis in $O_{\Delta_n}$.
\end{theo2}

\begin{proof}
In the proof of Lemma \ref{relations} we have explained how any tree monomial with given inputs and given output can be brought to what we have suggestively called its normal form. Now that we have fixed the monomial order, we see that abovementioned normal forms are indeed normal forms with respect to the relations $G$, and that these are the only normal forms: the applications of relations in the proof of Lemma \ref{relations} precisely correspond to lead-reducing in the those relations. Thus in every colored arity, the number of normal forms coincides with the dimension of the corresponding component of the operad (both being either 0 or 1), so $G$ is a Gr\"{o}bner basis according to Fact \ref{criterion}.
\end{proof}

Alternatively, we could have computed all the $S$-polynomials and shown that they reduce to $0$. Here is an example of such a computation. \\

For two relations both being of Type 1, their S-polynomial is defined when their leading terms intersect like this:

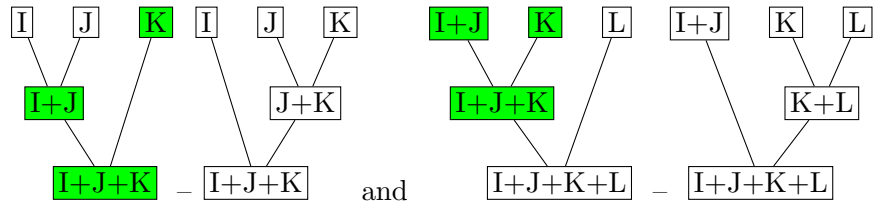
\begin{center}
 \makebox[20mm]{\begin{forest}
for tree = {grow'=90, draw, minimum width = 4pt, inner sep = 2pt, s sep = 15pt}
[I+J+K, fill = green [I+J, fill = green [I,tier = 1] [J]][K, tier = 1, fill = green]]
\end{forest}} --             \makebox[20mm]{\begin{forest}
for tree = {grow'=90, draw, minimum width = 4pt, inner sep = 2pt, s sep = 15pt}
[I+J+K [I, tier = 1][ J+K [J, tier = 1][ K, tier = 1]]]
\end{forest}}
and 
 \makebox[30mm]{\begin{forest}
for tree = {grow'=90, draw, minimum width = 4pt, inner sep = 2pt, s sep = 15pt}
[I+J+K+L [I+J+K, fill = green [I+J,tier = 1, fill = green] [K, fill = green]][L, tier = 1]]
\end{forest}} --             \makebox[25mm]{\begin{forest}
for tree = {grow'=90, draw, minimum width = 4pt, inner sep = 2pt, s sep = 15pt}
[I+J+K+L [I+J, tier = 1][ K+L [K, tier = 1][ L, tier = 1]]]
\end{forest}}
    \end{center}
    
their S-polynomial is this:

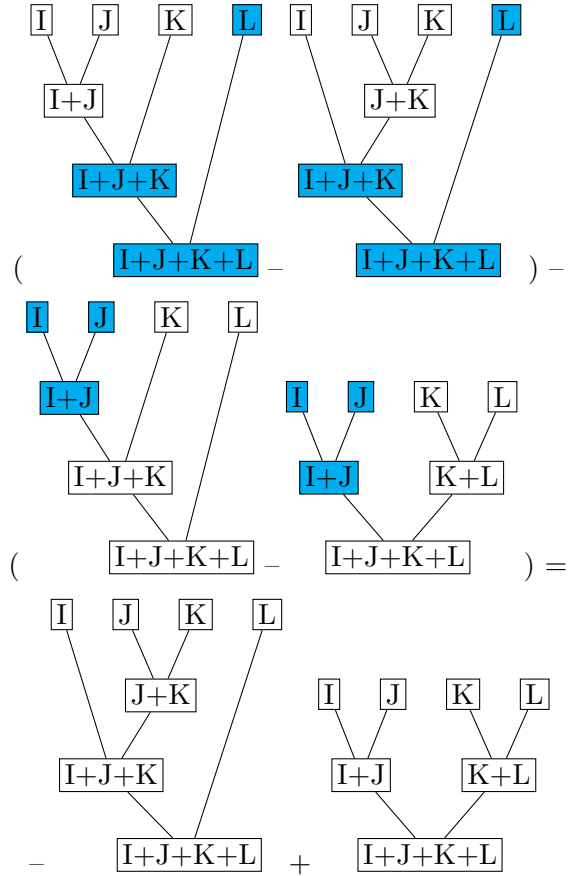
\begin{center}
 ( \makebox[30mm]{\begin{forest}
for tree = {grow'=90, draw, minimum width = 4pt, inner sep = 2pt, s sep = 15pt}
[I+J+K+L, fill = cyan [I+J+K, fill = cyan [I+J [I,tier = 1] [J]][K, tier = 1]] [L, tier = 1, fill = cyan]]
\end{forest}} --             \makebox[30mm]{\begin{forest}
for tree = {grow'=90, draw, minimum width = 4pt, inner sep = 2pt, s sep = 15pt}
[I+J+K+L, fill = cyan [I+J+K, fill = cyan [I, tier = 1][ J+K [J, tier = 1][ K, tier = 1]]] [L, fill = cyan, tier = 1]]
\end{forest}} ) --
\end{center}

\begin{center}
( \makebox[30mm]{\begin{forest}
for tree = {grow'=90, draw, minimum width = 4pt, inner sep = 2pt, s sep = 15pt}
[I+J+K+L [I+J+K [I+J, fill = cyan [I, fill = cyan] [J, tier = 1, fill = cyan]] [K, tier = 1]][L, tier = 1]]
\end{forest}} --             \makebox[30mm]{\begin{forest}
for tree = {grow'=90, draw, minimum width = 4pt, inner sep = 2pt, s sep = 15pt}
[I+J+K+L [I+J, fill = cyan [I, tier = 1, fill = cyan] [J, fill = cyan]][ K+L [K, tier = 1][ L, tier = 1]]]
\end{forest}} ) = 
\end{center}

\begin{center}
-- \makebox[30mm]{\begin{forest}
for tree = {grow'=90, draw, minimum width = 4pt, inner sep = 2pt, s sep = 15pt}
[I+J+K+L,  [I+J+K [I, tier = 1][ J+K [J, tier = 1][ K, tier = 1]]] [L, tier = 1]]
\end{forest}} 
+ 
\makebox[30mm]{\begin{forest}
for tree = {grow'=90, draw, minimum width = 4pt, inner sep = 2pt, s sep = 15pt}
[I+J+K+L [I+J [I, tier = 1] [J]][ K+L [K, tier = 1][ L, tier = 1]]]
\end{forest}} 
\end{center}

Reducing with respect to

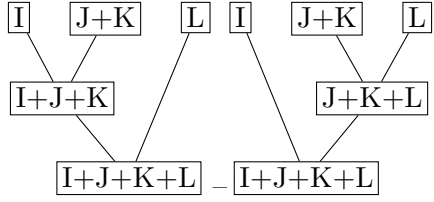
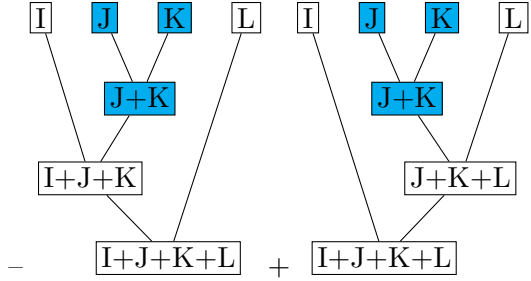
\begin{center}
 \makebox[25mm]{\begin{forest}
for tree = {grow'=90, draw, minimum width = 4pt, inner sep = 2pt, s sep = 15pt}
[I+J+K+L [I+J+K [I,tier = 1] [J+K]][L, tier = 1]]
\end{forest}} --             \makebox[25mm]{\begin{forest}
for tree = {grow'=90, draw, minimum width = 4pt, inner sep = 2pt, s sep = 15pt}
[I+J+K+L [I, tier = 1][ J+K+L [J+K, tier = 1][ L, tier = 1]]]
\end{forest}}
\end{center}

means subtracting

\begin{center}
 -- \makebox[30mm]{\begin{forest}
for tree = {grow'=90, draw, minimum width = 4pt, inner sep = 2pt, s sep = 15pt}
[I+J+K+L [I+J+K [I,tier = 1] [J+K, fill = cyan [J, fill = cyan] [K, tier = 1, fill = cyan] ]][L, tier = 1]]
\end{forest}} +             \makebox[30mm]{\begin{forest}
for tree = {grow'=90, draw, minimum width = 4pt, inner sep = 2pt, s sep = 15pt}
[I+J+K+L [I, tier = 1][ J+K+L [J+K, fill = cyan [J, fill = cyan] [K, tier = 1, fill = cyan]][ L, tier = 1]]]
\end{forest}}
\end{center}

so the reduction is equal to

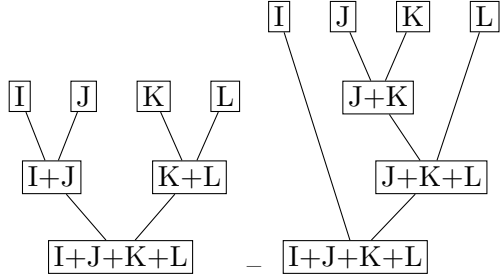
\begin{center}
\makebox[30mm]{\begin{forest}
for tree = {grow'=90, draw, minimum width = 4pt, inner sep = 2pt, s sep = 15pt}
[I+J+K+L [I+J [I, tier = 1] [J]][ K+L [K, tier = 1][ L, tier = 1]]]
\end{forest}} 
-- 
\makebox[30mm]{\begin{forest}
for tree = {grow'=90, draw, minimum width = 4pt, inner sep = 2pt, s sep = 15pt}
[I+J+K+L [I, tier = 1][ J+K+L [J+K [J] [K, tier = 1]][ L, tier = 1]]]
\end{forest}}
\end{center}

Further reducing with respect to 

    \begin{center}
 \makebox[30mm]
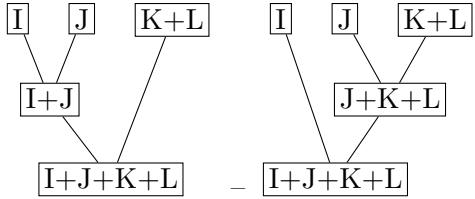
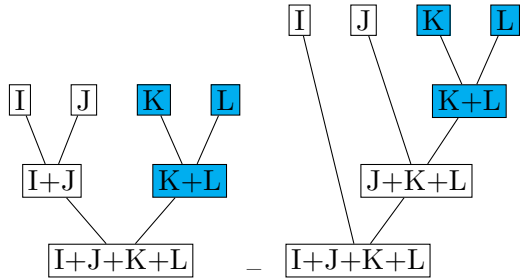
{\begin{forest}
for tree = {grow'=90, draw, minimum width = 4pt, inner sep = 2pt, s sep = 15pt}
[I+J+K+L [I+J [I,tier = 1] [J]][K+L, tier = 1]]
\end{forest}} --             \makebox[30mm]{\begin{forest}
for tree = {grow'=90, draw, minimum width = 4pt, inner sep = 2pt, s sep = 15pt}
[I+J+K+L [I, tier = 1][ J+K+L [J, tier = 1][ K+L, tier = 1]]]
\end{forest}}
    \end{center}
    
means subtracting 

    \begin{center}
 \makebox[30mm]{\begin{forest}
for tree = {grow'=90, draw, minimum width = 4pt, inner sep = 2pt, s sep = 15pt}
[I+J+K+L [I+J [I,tier = 1] [J]][K+L, fill = cyan [K, tier = 1, fill = cyan][L,fill = cyan]]]
\end{forest}} --             \makebox[35mm]{\begin{forest}
for tree = {grow'=90, draw, minimum width = 4pt, inner sep = 2pt, s sep = 15pt}
[I+J+K+L [I, tier = 1][ J+K+L [J, tier = 1][K+L, fill = cyan [K, tier = 1, fill = cyan][L,fill = cyan]]]]
\end{forest}}
    \end{center}

so the reduction is equal to

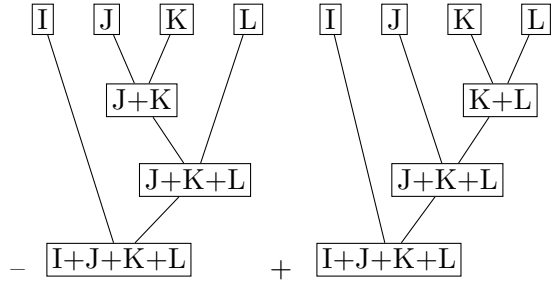
\begin{center}
-- \makebox[30mm]{\begin{forest}
for tree = {grow'=90, draw, minimum width = 4pt, inner sep = 2pt, s sep = 15pt}
[I+J+K+L [I, tier = 1][ J+K+L [J+K [J] [K, tier = 1]][ L, tier = 1]]] \end{forest}}
+
\makebox[35mm]{\begin{forest}
for tree = {grow'=90, draw, minimum width = 4pt, inner sep = 2pt, s sep = 15pt}
[I+J+K+L [I, tier = 1][ J+K+L [J, tier = 1][K+L [K, tier = 1][L]]]]
\end{forest}}
\end{center}

This clearly reduces to $0$ with respect to

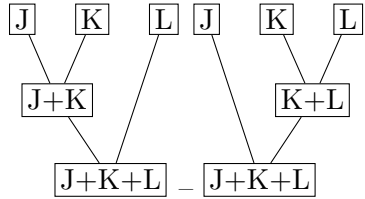
\begin{center}
 \makebox[20mm]{\begin{forest}
for tree = {grow'=90, draw, minimum width = 4pt, inner sep = 2pt, s sep = 15pt}
[J+K+L [J+K [J,tier = 1] [K]][L, tier = 1]]
\end{forest}} --             \makebox[20mm]{\begin{forest}
for tree = {grow'=90, draw, minimum width = 4pt, inner sep = 2pt, s sep = 15pt}
[J+K+L [J, tier = 1][ K+L [K, tier = 1][ L, tier = 1]]]
\end{forest}}
    \end{center}

\begin{cor2}
Operad $O_{\Delta_n}$ is Koszul.
\end{cor2}

We now compute its quadratic dual.

\begin{theo2}
$O_{\Delta_n}$ is quadratically self-dual.
\end{theo2}

\begin{proof}
The arities in which quadratic tree monomials are encountered are precisely the arities listed in  Lemma \ref{relations} (note that non-quadratic monomials are not encountered in those arities). In each of those arities, there are two possible tree monomials and one relation, equal to the difference of those two. So essentially the argument is the same as in proving that the associative operad is quadratically self-dual.
\end{proof}

\begin{cor2}
Operad $O_{\Delta_n}$ is Bar self-dual, and its Poincar\'e-Hilbert endomorphism satisfies the involutive property.
\end{cor2}

Note that a cube with standard directions is simply a product of several intervals. Thus the functoriality from the previous section also implies the statement for all cubes.

\begin{cor2}
For a cube $I^n$, its operad $O_{I^n}$ is Bar self-dual, and its Poincar\'e-Hilbert endomorphism satisfies the involutive property.
\end{cor2}

\section{Another example: polygons}
For polygons larger than the triangle, generating operations may be of arbitrary arity, but relations remain quadratic. So this section is structurally indistinguishable from the previous section, only the description of normal forms is a bit less pleasant. \\

Let $P$ be a polygon, with the source vertex labelled by $0 = x(0) = y(0)$, the sink vertex labelled by $1 = x(n+1) = y(m+1)$, vertices of the upper path labelled by $x(1)$ to $x(n)$, edges of the upper path labelled by $e(0)$ to $e(n)$, vertices of the lower path labelled by $y(1)$ to $y(m)$ and edges of the lower path labelled by $f(0)$ to $f(m)$: 

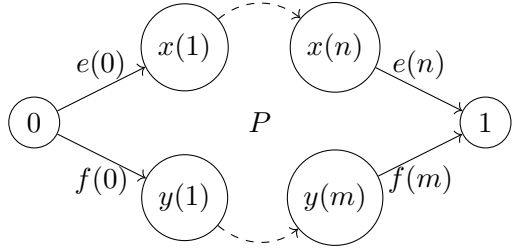
\begin{center}
\begin{tikzpicture}

\node (P) at (3,0) {$P$};

\node[circle, draw] (0) at (0,0) {0};
\node[circle, draw] (x1) at (2,1) {$x(1)$};
\node[circle, draw] (xn) at (4,1) {$x(n)$};
\node[circle, draw] (1) at (6,0) {$1$};

\node[circle, draw] (y1) at (2,-1) {$y(1)$};
\node[circle, draw] (ym) at (4,-1) {$y(m)$};

\draw[->] (0) -- (x1) node[midway, above] {$e(0)$};
\draw[->] (xn) -- (1) node[midway, above] {$e(n)$};

\draw[->] (0) -- (y1) node[midway, below] {$f(0)$};
\draw[->] (ym) -- (1) node[midway, below] {$f(m)$};

\draw[->, dashed] (x1) to[out=40,in=140] (xn);
\draw[->, dashed] (y1) to[out=-40,in=-140] (ym);

\end{tikzpicture}
\end{center}

\begin{lem2}
\label{generators2}
Generating operations for $O_{P}$ are of the following types:

\begin{enumerate}
    \item left vertex inclusions $U_{\lef}(x(i))$ and $U_{\lef}(y(i))$:
    
    \begin{center}

\makebox[20mm]{\begin{forest}
for tree = {grow'=90, draw, minimum width = 4pt, inner sep = 2pt, s sep = 15pt}
[e(i)  [x(i)]]
\end{forest}}
\makebox[20mm]{\begin{forest}
for tree = {grow'=90, draw, minimum width = 4pt, inner sep = 2pt, s sep = 15pt}
[f(i)  [y(i)]]
\end{forest}}
\end{center}

    \item right vertex inclusions $U_{\rig}(x(i))$ and $U_{\rig}(y(i))$:
\begin{center}
\makebox[20mm]{\begin{forest}
for tree = {grow'=90, draw, minimum width = 4pt, inner sep = 2pt, s sep = 15pt}
[e(i-1)  [x(i)]]
\end{forest}} 
\makebox[20mm]{\begin{forest}
for tree = {grow'=90, draw, minimum width = 4pt, inner sep = 2pt, s sep = 15pt}
[f(i-1)  [y(i)]]
\end{forest}} 
\end{center}

\item vertex-vertex actions $B(x(i),x(i))$ and $B(y(i),y(i))$:

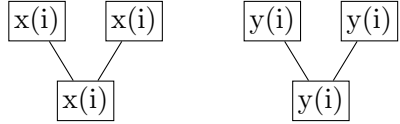
\begin{center}
\makebox[30mm]{\begin{forest}
for tree = {grow'=90, draw, minimum width = 4pt, inner sep = 2pt, s sep = 15pt}
[x(i)  [x(i)] [x(i)] ]
\end{forest}}
\makebox[30mm]{\begin{forest}
for tree = {grow'=90, draw, minimum width = 4pt, inner sep = 2pt, s sep = 15pt}
[y(i)  [y(i)] [y(i)] ]
\end{forest}}
\end{center}

    \item vertex-edge actions $B(x(i),e(i))$ and $B(y(i),f(i))$:
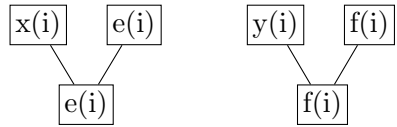
    \begin{center}
\makebox[30mm]{\begin{forest}
for tree = {grow'=90, draw, minimum width = 4pt, inner sep = 2pt, s sep = 15pt}
[e(i)  [x(i)] [e(i)] ]
\end{forest}}
\makebox[30mm]{\begin{forest}
for tree = {grow'=90, draw, minimum width = 4pt, inner sep = 2pt, s sep = 15pt}
[f(i)  [y(i)] [f(i)] ]
\end{forest}}
\end{center}
    \item edge-vertex actions $B(e(i-1),x(i))$ and $B(f(i-1),y(i))$:

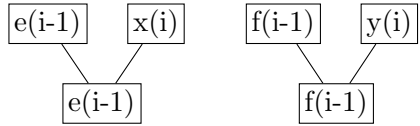
\begin{center}
\makebox[30mm]{\begin{forest}
for tree = {grow'=90, draw, minimum width = 4pt, inner sep = 2pt, s sep = 15pt}
[e(i-1)  [e(i-1) ] [x(i)] ]
\end{forest}} 
\makebox[30mm]{\begin{forest}
for tree = {grow'=90, draw, minimum width = 4pt, inner sep = 2pt, s sep = 15pt}
[f(i-1) [f(i-1)] [y(i)]]
\end{forest}} 
\end{center}
    \item 0-action $B(0,P)$:
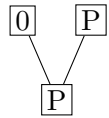
\begin{center}
    \makebox[20mm]{\begin{forest}
for tree = {grow'=90, draw, minimum width = 4pt, inner sep = 2pt, s sep = 15pt}
[P  [0] [P]]
\end{forest}}
\end{center}
    \item 1-action $B(P,1)$:
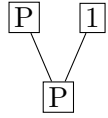
\begin{center}
    \makebox[20mm]{\begin{forest}
for tree = {grow'=90, draw, minimum width = 4pt, inner sep = 2pt, s sep = 15pt}
[P  [P] [1]]
\end{forest}}
\end{center}
    
\item edge-sequences (note that the sequence of edges may be of length 1, to account for edge inclusions, or may be disconnected) $M(e(i_1), \ldots, e(i_k))$ and $M(f(i_1), \ldots, f(i_k))$ for $i_1 < \ldots < i_k$:
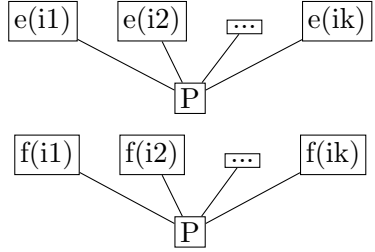
\begin{center}
\makebox[20mm]{\begin{forest}
for tree = {grow'=90, draw, minimum width = 4pt, inner sep = 2pt, s sep = 15pt}
[P  [e(i1)] [e(i2)] [...] [e(ik)]]
\end{forest}}
\end{center}
\begin{center}
\makebox[20mm]{\begin{forest}
for tree = {grow'=90, draw, minimum width = 4pt, inner sep = 2pt, s sep = 15pt}
[P  [f(i1)] [f(i2)] [...] [f(ik)]]
\end{forest}}
\end{center}
    
\end{enumerate}
\end{lem2}

\begin{proof}
We only have to deal with operations that have output $P$ (operations with other outputs are covered by considerations for simplices).  Consider a nonzero operation whose inputs are the sequence $\sigma$, consisting of: several (maybe none) instances of $0$, then some sequence of edges and vertices along the upper path, then several (maybe none) instances of $1$. For every vertex $v$ along the upper path, we define its subword $\sigma_v$ as follows: 
\begin{itemize}
    \item $\sigma_{x(i)}$ for $i<n$ consists of all instances of $x(i)$ in $\sigma$, and of $e(i)$ if $e(i)$ is in $\sigma$
    \item $\sigma_{x(n)}$ consists of all instances of $x(n)$ in $\sigma$, of $e(n)$ whenever it is in $\sigma$, and of all instances of $1$ in $\sigma$
    \item $\sigma_1$ is empty when $\sigma_{x(n)}$ is nonempty; otherwise it consists of $e(n)$ whenever it is in $\sigma$ and of all instances of $1$ in $\sigma$
\end{itemize}

Thus $\sigma$ is represented as a concatenation of subwords $\sigma_v$ (ignore the empty ones). For every subword $\sigma_{x(i)}$ with $0 < i < n$, we apply generators in one of the following ways, depending on whether $e(i)$ is already in $\sigma_{x(i)}$ or not, to obtain an operation with output $e(i)$:

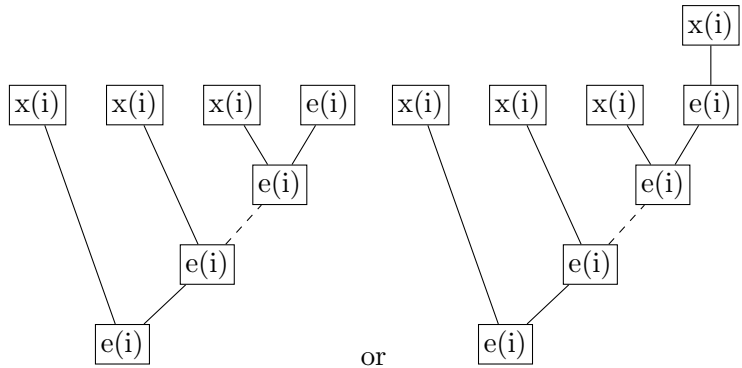
\begin{center}
 \makebox[45mm]{\begin{forest}
for tree = {grow'=90, draw, minimum width = 4pt, inner sep = 2pt, s sep = 15pt}
[ e(i) [x(i), tier = 1][e(i) [x(i), tier = 1][e(i), edge = dashed [x(i)][e(i), tier = 1]]]]
\end{forest}} or             \makebox[45mm]{\begin{forest}
for tree = {grow'=90, draw, minimum width = 4pt, inner sep = 2pt, s sep = 15pt}
[ e(i) [x(i), tier = 1][e(i) [x(i), tier = 1][e(i), edge = dashed [x(i)][e(i), tier = 1 [x(i)]]]]]
\end{forest}}
\end{center}
    
For the subword $\sigma_{x(n)}$ (or for $\sigma_1$, depending on which one is nonempty) we apply generators in one of the following ways, depending on whether $e(n)$ is already in the subword or not:

\begin{center}
 \makebox[40mm]{\begin{forest}
for tree = {grow'=90, draw, minimum width = 4pt, inner sep = 2pt, s sep = 15pt}
[e(n),[x(n), tier = 1] [e(n), edge = dashed [e(n), tier = 1] [1 [1] [1, edge = dashed, tier = 1]]] ]
\end{forest}} or             \makebox[40mm]{\begin{forest}
for tree = {grow'=90, draw, minimum width = 4pt, inner sep = 2pt, s sep = 15pt}
[e(n),[x(n), tier = 1] [e(n), edge = dashed [e(n), tier = 1 [1]] [1 [1] [1, edge = dashed, tier = 1]]] ]
\end{forest}}
\end{center}

Then we apply an operation with inputs $e(0)$ (if it is in $\sigma$) and  $ \{ e(i_s) \}$ (where $e(i_s)$ are outputs of the operations described above), and with output $P$. And finally, we apply $0$-action as many times as needed to deal with $\sigma_0$. As with the simplices, this particular tree monomial is suggestively called the normal form (which it will be, after we fix the monomial order and write out the relations). \\

The case of lower path is identical, after replacing $x$ with $y$ and $e$ with $f$.
\end{proof}

\begin{eg2}
Consider the case of $n=2$ and $\sigma = (0,0,e(0),x(1),x(1),x(2),e(2),1)$. Then the subwords are $\sigma_0 = (0,0,e(0))$, $\sigma_{x(1)} = (x(1),x(1))$, $\sigma_{x(2)} = (x(2),e(2),1)$, and the corresponding normal form is this:
\begin{center}
\begin{forest}
for tree = {grow'=90, draw, minimum width = 4pt, inner sep = 2pt, s sep = 15pt}
[P [0, tier = 1] [ P [0]  [P, tier = 1 [e(0)] [ e(1) [x(1)] [e(1)[x(1), tier = 2]]] [ e(2) [x(2), tier = 2] [e(2) [e(2)] [1]]] ]]]
\end{forest}
\end{center}

Informally, all of this was to say: for the normal form we choose the most right-leaning tree of all.
\end{eg2}

In the following lemma we list the (straightforward though tiresome to write) relations arising on quadratic tree monomials.

\begin{lem2}
\label{relationsp}
Relations $G$ on the generating operations are of the following types:
\end{lem2}

\begin{enumerate}
    \item generalized simple associativities, i.e. relations of the type 
    $$ B(?,B(?,?)) - B(B(?,?),?)$$
    where arity is one of the following:
    \begin{itemize}
        \item  $(x(i)^3;x(i))$ or $(y(i)^3;y(i))$
        \item  $(x(i)^2,e(i);e(i))$ or $(y(i)^2,f(i);f(i))$
        \item $(x(i),e(i),x(i+1);e(i))$ or $(y(i),f(i),y(i+1);f(i))$
        \item $(e(i),x(i+1)^2;e(i))$ or $(f(i),y(i+1)^2;f(i))$
        \item $(0^2,P;P)$
         \item $(P,1^2;P)$
    \end{itemize}

    \item in arity $(0, e(0), e(i_1), \ldots, e(i_s); P)$ or  $(0, f(0), f(i_1), \ldots, f(i_s); P)$  for $0<i_1<\ldots<i_s$:
    $$ B \big( 0, M \left( e(0), e(i_1), \ldots, e(i_s) \right) \big) - M(B(0,e(0)),e(i_1), \ldots, e(i_s)) $$
    or 
    $$ B(0, M( f(0), f(i_1), \ldots, f(i_s))) - M(B(0,f(0)),f(i_1), \ldots, f(i_s)) $$
    
    \item in arity $(0 = x(0), e(i_1), \ldots, e(i_s); P)$  or $(0 = y(0), f(0), \ldots, f(i_s); P)$ for $0<i_1<\ldots<i_s$:
    $$ B(x(0), M( e(i_1), \ldots, e(i_s))) - M(U_{\lef}(x(0)),e(i_1), \ldots, e(i_s)) $$
    or 
    $$ B(y(0), M( f(i_1), \ldots, f(i_s))) - M(U_{\lef}(y(0)),f(i_1), \ldots, f(i_s)) $$
    
    \item in arity $(e(i_1), \ldots, e(i_s), e(n), 1; P)$ or  $(f(i_1), \ldots, f(i_s), f(n), 1; P)$  for $i_1<\ldots<i_s < n$:
    $$ M(e(i_1), \ldots, e(i_s), B(e(n), 1)) - B(M(e(i_1), \ldots, e(i_s), e(n)), 1) $$
    or 
    $$ M(f(i_1), \ldots, e(f_s), B(f(n), 1)) - B(M(f(i_1), \ldots, f(i_s), f(n)), 1) $$
    
    \item in arity $(e(i_1), \ldots, e(i_s), 1 = x(n+1); P)$ or  $(f(i_1), \ldots, f(i_s), 1 = y(m+1); P)$  for $i_1<\ldots<i_s < n$:
    $$ M(e(i_1), \ldots, e(i_s), U_{\rig}(x(n+1)) - B(M(e(i_1), \ldots, e(i_s)), x(n+1))$$ 
    or 
    $$ M(f(i_1), \ldots, f(i_s), U_{\rig}(y(m+1)) - B(M(f(i_1), \ldots, f(i_s)), y(m+1))$$  
    
    \item in arity $(e(i_1), \ldots, e(i_a), x(j),  e(i_{a+1}), \ldots, e(i_s); P)$ or \linebreak $(f(i_1), \ldots, f(i_a), y(j),  f(i_{a+1}), \ldots, f(i_s); P)$  for $i_1< \ldots < i_s$ and $i_{a}+1 < j < i_{a+1}$:
    \begin{align*}
& M(e(i_1), \ldots, e(i_a), U_{\lef}(x(j)),  e(i_{a+1}), \ldots, e(i_s)) - \\
& M(e(i_1), \ldots, e(i_a), U_{\rig}(x(j),  e(i_{a+1}), \ldots, e(i_s)) 
    \end{align*}
    or
    \begin{align*}
& M(f(i_1), \ldots, f(i_a), U_{\lef}(y(j),  f(i_{a+1}), \ldots, f(i_s))- \\
& M(f(i_1), \ldots, f(i_a), U_{\rig}(y(j),  f(i_{a+1}), \ldots, f(i_s)) 
    \end{align*}
    
    \item in arity $(e(i_1), \ldots, e(i_a), x(i_a),  e(i_{a+1}), \ldots, e(i_s); P)$ or \linebreak $(f(i_1), \ldots, f(i_a), y(i_a),  f(i_{a+1}), \ldots, f(i_s); P)$ for $i_1< \ldots < i_s$:
    
    \begin{align*}
& M(e(i_1), \ldots, B(e(i_a),x(i_a)),  e(i_{a+1}), \ldots, e(i_s)) - \\
& M(e(i_1), \ldots, e(i_a), U_{\lef}(x(i_a)),  e(i_{a+1}), \ldots, e(i_s)) 
    \end{align*}
    or
    \begin{align*}
& M(f(i_1), \ldots, B(f(i_a),y(i_a)),  f(i_{a+1}), \ldots, f(i_s))- \\
& M(f(i_1), \ldots, f(i_a), U_{\lef}(y(i_a)),  f(i_{a+1}), \ldots, f(i_s)) 
    \end{align*}
    
    \item in arity $(e(i_1), \ldots, e(i_a), x(i_{a+1}+1),  e(i_{a+1}), \ldots, e(i_s); P)$ or \linebreak $(f(i_1), \ldots, f(i_a), y(i_{a+1}+1),  f(i_{a+1}), \ldots, f(i_s); P)$ for $i_1< \ldots < i_s$:
    
    \begin{align*}
& M(e(i_1), \ldots, B(x(i_a+1), e(i_{a+1})), \ldots, e(i_s)) - \\
& M(e(i_1), \ldots, e(i_a), U_{\rig}(x(i_a+1)),  e(i_{a+1}), \ldots, e(i_s)) 
    \end{align*}
    or
    \begin{align*}
& M(f(i_1), \ldots, B(y(i_a),f(i_{a+1})), \ldots, f(i_s))- \\
& M(f(i_1), \ldots, U_{\rig}(y(i_a+1)),  f(i_{a+1}), \ldots, f(i_s)) 
    \end{align*}
    
    \item in arity $(e(i_1), \ldots, e(i_a), x(i_a+1),  e(i_{a+1}), \ldots, e(i_s); P)$ or \linebreak $(f(i_1), \ldots, f(i_a), y(i_a+1),  f(i_{a+1}), \ldots, f(i_s); P)$ for $i_1< \ldots < i_s$ and $i_{a}+1 = i_{a+1}$:
    
    \begin{align*}
& M(e(i_1), \ldots, B(e(i_a), x(i_a+1)), \ldots, e(i_s)) - \\
& M(e(i_1), \ldots, B(x(i_a+1), e(i_{a+1})), \ldots, e(i_s))
    \end{align*}
    or
    \begin{align*}
& M(f(i_1), \ldots, B(f(i_a), y(i_a+1)), \ldots, f(i_s)) - \\
& M(f(i_1), \ldots, B(y(i_a+1), f(i_{a+1})), \ldots, f(i_s))
    \end{align*}
    
\end{enumerate}

\begin{proof}
As before, we show that these relations are sufficient to bring any tree monomial to its normal form described in the proof of Lemma \ref{generators2}. We delay the description of this procedure until we have defined the monomial order, so that we could simultaneously prove that $G$ is a Gr\"{o}bner basis.
\end{proof}

Similarly to the previous section, we derive path-lexicographic order on all tree monomials from the following order on generating operations: 

\begin{itemize}
    \item operation of a smaller arity is smaller 
    \item for unary operations, $U_{\rig}<U_{\lef}<M(e) $
    \item for binary operations, we first compare outputs: vertices are smaller than edges, edges are smaller than $P$, vertices and edges are ordered left to right; if outputs coincide, we compare the input sequences lexicographically.
    \item for operations $M(\sigma)$ and $M(\sigma')$ with $|\sigma|=|\sigma'|>2$ the sequences $\sigma$ and $\sigma'$ are compared lexicographically.
    
\end{itemize}

\begin{theo2}
Relations from Lemma \ref{relationsp} form a quadratic Gr\"{o}bner basis of $\mathcal{O}_P$.
\end{theo2}
\begin{proof}
Similarly to the case of simplices, we observe that normal forms described above are irreducible with respect to $G$ and are the only tree monomials of given arities with this property, thus the number of normal forms coincides with the dimension of the respective operadic component. Indeed, let $T$ be a tree monomial of arity $(\sigma;P)$, with $\sigma$ being a sequence as in the proof of Lemma \ref{generators2}. \\

\begin{enumerate}
    \item By lead-reducing in relations of types 1, 2 and 3 we can bring $T$ to the form 
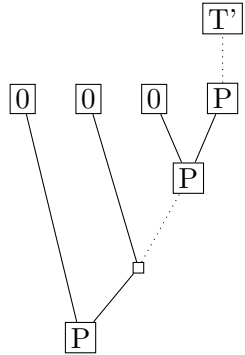
\begin{center}
\begin{forest}
for tree = {grow'=90, draw, minimum width = 4pt, inner sep = 2pt, s sep = 15pt}
[P [0, tier = 1] [[0, tier = 1] [P, edge = dotted [0, tier = 1] [P [T', edge = dotted]]]]]
\end{forest}
\end{center}
where rooted at $P$, there is a tree monomial $T'$ with its inputs $\sigma'$ devoid of $0$. We now work with $T'$.

\item By lead-reducing in relations of types 4 and 5, we bring $T'$ to the form where the bottom operation is $M$: 

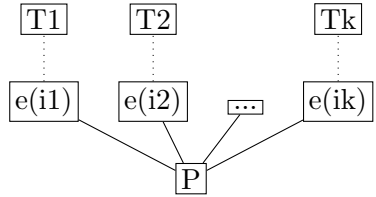
\begin{center}
\begin{forest}
for tree = {grow'=90, draw, minimum width = 4pt, inner sep = 2pt, s sep = 15pt}
[P [e(i1) [{T1}, edge = dotted]] [e(i2) [{T2}, edge = dotted]] [...] [e(ik) [{Tk}, edge  = dotted]]] 
\end{forest}
\end{center}

\item By lead-reducing in relations of types $6$, $7$, $8$ and $9$, we ensure that every nonterminal vertex appearing in $\sigma$ belongs to the subtree $T_{l}$ whose root is an edge $e(i_l)$ to the right of this vertex.

\item Finally, by read-reducing again in relations of type 1, we ensure that trees $T_i$ are right-leaning.
\end{enumerate}

\end{proof}

\begin{cor2}
The operad $\mathcal{O}_P$ is Koszul.
\end{cor2}

\begin{theo2}
The operad $\mathcal{O}_P$ is quadratically self-dual.
\end{theo2}
\begin{proof}
Similarly to the case of the simplices, the arities in which quadratic tree monomials are encountered are precisely the arities listed in  Lemma \ref{relationsp}, with no non-quadratic tree monomials in these arities. In each of these arities, there are two possible tree monomials and one relation, equal to the difference of those two, so again the argument is a generalization of the argument for self duality of the associative operad.
\end{proof}

\begin{cor2}
The operad $\mathcal{O}_P$ is Bar self-dual, and its Poincar\'e-Hilbert endomorphism satisfies the involutive property.
\end{cor2}

\section{Shortness and integrated $A_\infty$-coalgebras}
\label{algebra}
We now explain the generality in which we expect the self-duality conjecture to hold.

\begin{defi2}
A directed polytope $P$ is called short if nontrivial chains have positive excesses.
\end{defi2}

In terms of operads, shortness has a very precise meaning: for a short polytope $P$, its operad $\mathcal{O}_P$ is nonnegatively graded with respect to the {\em total} grading, with $O^0_P$ being just the semisimple algebra $R = \bigoplus k \cdot 1_F$: in other words, the operad is {\em augmented}.

\begin{con2}
\label{conjecture}
For short polytopes their operads are Koszul and Koszul self-dual.
\end{con2}

Examples of short polytopes include simplices and polygons. Products of short polytopes are also short, thus adding cubes into this class.  In a sequel to this paper we prove shortness for freehedra.\\

Our evidence for the Conjecture \ref{conjecture}, besides proved results for simplices and polygons, includes Sage verification of the involutive property of the Poincar\'e-Hilbert endomorphism of $\mathcal{O}_P$ for $P$ several non-standard directions of 3D cube, pyramid, octahedron and 3D freehedron. The involutive property held for all short polytopes and failed for all non-short polytopes. \\

Now suppose that there is some magician to fix all the signs for us. Formally, let us work over $\mathbb{F}_2$ (although we expect our theory to work generally). Let $P$ be a short polytope, and let $C_*(P)$ be the {\em graded vector space} of cellular chains. Consider the degree $k$ maps 

$$\Delta_{n}^{k}: C_*(P)[1] \to C_*(P)[1]^{\otimes{n}}$$
$$ F \mapsto \sum _{(F_1, \ldots, F_m) \in \fc_k(F,n)} F_1 \otimes \ldots \otimes F_n$$

where the face generator $F$ maps to the sum of tensor products of chains in $F$ that have length $n$ and excess $k$. For example, $\Delta_1^0$ is the identity map and $\Delta_1^{1}$ is the cellular differential. A more interesting observation: for the case of simplices, $\Delta_2^1$ is the Alexander-Whitney diagonal, and for the case of cubes, $\Delta_2^1$ is the Serre diagonal. \\

The following (elementary but surprising) theorem explains the connection of the theory built in the present paper with the $A_\infty$-world (see \cite{4Kel} for exposition).

\begin{theo2}
Shortness condition and self-duality of $\mathcal{O}_P$ together imply that operations $\Delta_{n}^{1}$ for $n \geq 1$ assemble into $A_\infty$-coalgebra structure on $C_*(P)$.
\end{theo2}
\begin{proof}
The Poincar\'e-Hilbert endomorphism can be written as a $t$-linear endomorphism of $\widehat{T}(C_*(P)[1])[[t]]$ which is, on generators, given by
$$ \operatorname{Id} + \Delta^1 t + \Delta^2 t^2 + \ldots $$
where $\Delta^k$ is the sum $\sum_{n \geq 1} \Delta^k_n$.
The involutive property says that, modulo signs, this is an involution. So if we extend $\Delta^1$ from generators as a derivation, it would square to $0$ (note that we use working modulo $2$ when we say $\operatorname{Id}\Delta^2 + \Delta^2 \operatorname{Id} = 0$). This is precisely the compact definition of $A_\infty$-relations.
\end{proof}

\begin{rem2}
For the theory to work with signs, we certainly need to replace identity with parity; putting signs in order is a work in progress.
\end{rem2}

This motivates the following definition.\\

\begin{defi2}
The structure of an integrated $A_\infty$-coalgebra
 on a graded vector space $V$ is the data of $t$-linear involution on the completed tensor algebra $\widehat{T}(V[1])$ \end{defi2}

Our original goal was to apply this machinery for associahedra with Tamari directions. Unfortunately, in dimensions $\geq 4$ they fail to be short (though all chains of length $2$ do have positive excesses, thus allowing the Saneblidze-Umble diagonal \cite{4SU} to satisfy Leibniz rule). A further direction of our research is to develop a modification of $\mathcal{O}_P$-construction that would provide an augmented Koszul self-dual operad for associahedra.\\

The idea for this modification is to reduce the class of admissible face chains. The hint for this reduction is contained in the diagonal for permutahedra. In permutahedra, chains $F_1 < F_2$ of length 2 may have nonpositive excess. However, as explained in \cite{L-A}, there is a stronger condition $F_1 << F_2$ which fails for all chains of nonpositive excess, and some of their subchains. In is an open question to formulate higher analogues of this condition that will outrule longer chains with nonpositive excess.

\end{document}